\documentclass{amsart}
\usepackage{amsmath}
\usepackage{graphicx}
\usepackage{amsfonts}
\usepackage{amssymb}
\usepackage{pdfsync}
\usepackage{color}
\usepackage{mathrsfs}
\usepackage{epsfig}
\usepackage{url}
\usepackage{mathrsfs,a4wide}
\usepackage{url}
\usepackage{calligra}
\usepackage{mathtools}
\usepackage{esint}

\makeatletter
\def\@splitop#1#2\@nil{$\mathscr{#1}\!\!$\calligra#2\,\,}
\newcommand*\DeclareCursiveOperator[2]{%

  \newcommand#1{\mathop{\mbox{\@splitop#2\@nil}}\nolimits}}
\makeatother
\DeclareCursiveOperator{\Bnew}{B}
\DeclareCursiveOperator{\Cnew}{C}
\DeclareCursiveOperator{\Dnew}{D}
\DeclareCursiveOperator{\defe}{Def}

\theoremstyle{plain}
\newtheorem{theorem}{Theorem}[section]
\newtheorem{lemma}[theorem]{Lemma}
\newtheorem{proposition}[theorem]{Proposition}

\newtheorem{remark}[theorem]{Remark}

\theoremstyle{definition}
\numberwithin{equation}{section}

\newcommand{\rs}{{\mathcal R}}
\newcommand{\sss}{{\mathcal S}}

\newcommand{\E}{{\mathscr E}}

\newcommand{\F}{{\mathscr F}_L}

\newcommand{\G}{\mathscr{G}}

\newcommand{\R}{{\mathbb R}}
\newcommand{\N}{{\mathbb N}}
\newcommand{\Z}{{\mathbb Z}}

\newcommand{\C}{{\mathbb C}}

\newcommand{\ud}{\mathrm{d} }

\newcommand{\ep}{\varepsilon}

\newcommand{\newatop}{\genfrac{}{}{0pt}{1}}

\newcommand{\loc}{{\mathrm{loc}}}
\newcommand{\X}{\mathcal{X}}

\newcommand{\XX}{\mathrm{X}}

\newcommand{\res}{\mathop{\hbox{\vrule height 7pt width .5pt depth 0pt
\vrule height .5pt width 6pt depth 0pt}}\nolimits}

\title
[Periodic ground states for nonlocal energies]{Periodic ground states  
for a
one-parameter family of 
nonlocal energies on the real line}

\author[L. De Luca]
{Lucia De Luca}
\address[L. De Luca]{Istituto per le Applicazioni del Calcolo ``M. Picone'', IAC-CNR, via dei Taurini 19, 00185 Roma, Italy}
\email[L. De Luca]{lucia.deluca@cnr.it}

\author[G. Pini]
{Giovanni Pini}
\address[G. Pini]{Dip. di Matematica, Univ-Roma I ``La Sapienza'', Piazzale Aldo Moro 5, 00185 Roma, Italy}
\email[G. Pini]{giovanni.pini@uniroma1.it}

\author[M. Ponsiglione]
{Marcello Ponsiglione}
\address[M. Ponsiglione]{Dip. di Matematica, Univ-Roma I ``La Sapienza'', Piazzale Aldo Moro 5, 00185 Roma, Italy}
\email[M. Ponsiglione]{ponsigli@mat.uniroma1.it}

\author[F. Santilli]
{Francesco Santilli}
\address[F. Santilli]{Dip. di Matematica, Univ-Roma I ``La Sapienza'', Piazzale Aldo Moro 5, 00185 Roma, Italy}
\email[F. Santilli]{francesco.santilli@mat.uniroma1.it}

\usepackage{hyperref}
\begin{document}
\vskip .2truecm

\begin{abstract}
\small{In dimension one, we introduce a one-parameter family of nonlocal  energies, including subcritical Gagliardo seminorms and Riesz functionals, 
acting on scalar functions whose derivative is given by a periodic configuration of screened Dirac masses. 
We prove that the global minimizers are given by the equi-spaced configurations of particles. 

This result is achieved by representing the energy functionals as interaction potentials depending on the geodesic  distances between the particles, exploiting the complete monotonicity properties of such potentials and invoking the celebrated Cohn-Kumar Universality Theorem.

\vskip .3truecm \noindent \textsc{Keywords}: Calculus of Variations; Fractional Seminorms;  Periodic Minimizers
\vskip.1truecm \noindent \textsc{AMS Mathematics Subject Classification}: 
49J45, 82C22, 52C35



}
\end{abstract}
\maketitle

{\small \tableofcontents}

\section*{Introduction}
This paper deals with the minimization of nonlocal energies defined on periodic scalar potentials of one variable, associated with screened particles lying on $\R$. Problems of this kind have attracted considerable attention in the mathematical community, and several specific energies have been investigated in connection with applications to crystallization, Materials Science, and gas dynamics. 

The aim of this paper is to identify a natural one-parameter family of energy functionals that encompasses well-studied models as well as new and often overlooked ones, and to provide a unified and concise proof of the periodicity of minimizers.

Given $L\in\N$, we consider $L$-periodic empirical measures $\mu$ given by the sum of exactly $L$ Dirac masses on the interval $[0,L)$. To these measures we associate admissible periodic functions $u$ satisfying $\mathrm{D} u=\ud x -\mu$ on $[0,L)$,  namely, the particles are screened by the Lebesgue measure.
We then introduce a one-parameter family of energy functionals $\E^\sigma$ (see \eqref{fraclapi}), where the parameter $\sigma$ ranges in the interval $(-\frac{1}{2}, \frac{1}{2})$. For positive values of $\sigma$, these energies are Gagliardo-type seminorms; for negative values of $\sigma$, they are Riesz-type functionals, that may be seen as the continuous long-range counterpart of  XY energies defined on densities. Although the latter are less canonical in variational problems, they appear to be the natural extension of Gagliardo seminorms to negative values of $\sigma$. These two branches of energy functionals meet, after suitable scaling, at the case $\sigma=0$, which is simply the squared $L^2$-norm \cite{MS}. In this respect, the considered family enjoys natural continuity properties with respect to the parameter $\sigma$ \cite{DNP, CDKNP, DMPS, PS}, in terms of $\Gamma$-convergence, and has recently been identified as a natural one-parameter family to be studied as a whole in connection with geometric evolutions \cite{CDNP} and parabolic problems \cite{CDKNP, DMPS}.

In this paper, we consider the same family of functionals in the context of one-dimensional 
periodic distributions of screened particles. Again, while the case $\sigma>0$ has been extensively studied in the literature \cite{GM,  DPS, DGP}, its natural extension to negative values of $\sigma$ appears to be completely novel (see \cite{GLL} for a related model). Remarkably, the local case $\sigma=0$ has been a pivotal case study in this theory \cite{M93} (see also \cite{RW}), while the critical case $\sigma=\frac{1}{2}$ (not considered in this paper) has attracted much attention because of its relevance in Materials Science \cite{GM,  DPS, DGP}, gas dynamics \cite{SS, L} and random matrices \cite{BS}.

Our strategy is to exploit Fourier analysis to rewrite the energy functionals in terms of pairwise interactions defined directly on the empirical measures. Assuming periodicity, we deal with $L$ points lying on $\mathbb S^1$; 
then we disintegrate this interaction potential into a family $\{V_t\}_{t>0}$ of strictly completely monotone functions depending on the squared Euclidean distances between the particles on $\mathbb S^1$, 
thus
 fitting into the general Universality Theory of Cohn--Kumar \cite{CK}. Such a theory ensures that, among regular configurations without superpositions of particles (which have finite energy since $\sigma < \frac{1}{2}$), there exists a minimizer given by the periodic distribution of particles lying at the vertices of a regular $L$-gon, and that such a configuration is the unique minimizer up to rotations. We then prove that these are in fact the only minimizers, even when collisions of particles are allowed. Indeed,  on the one hand, optimal configurations minimize the  energy functional governed by $V_t$ for every  $t>0$; on the other hand, as $t$ vanishes, the  potentials $V_t$ blow up as soon as (at least) two   particles approach each other, thereby ruling out collisions for sufficiently small values of $t$.

{\vskip5pt
{\textsc{Acknowledgements:} LDL, GP, and FS are members of the Gruppo Nazionale per l'Analisi Matematica, la Probabilit\`a e le loro Applicazioni (GNAMPA) of the Istituto Nazionale di Alta Matematica (INdAM).
 
}

%

\section{Setting of the problem}\label{Suno}
In this section we introduce the energy functionals and the class of admissible functions we will deal with. 

\subsection{Configurations}  
Let $L \in\N$ be fixed.
We set
$
\X^{L}:= [0,L)^L.
$
Notice that the entries  of a generic configuration $X\in\X^{L}$ are not required to be distinct. 

For every $X=(x_1,\ldots,x_L)\in \X^{L}$ we define 
$$\sss(X):=\{\xi \in[0,L)\,:\,x_ l = \xi \textrm{ for some } l=1,\ldots,L\},
$$
and we denote by $\rs^L$ the set of configurations  for which all the entries are distinct, i.e.,
$$
\rs^L:=\{X\in\X^L\,:\,\sharp\sss(X)=L\}.
$$
%
For every $X=(x_1,\ldots, x_L)\in\X^L$  we denote by $\XX$ its {\it $L$-periodic extension}, namely the $L$-periodic sequence $\XX:= \{\mathsf x_z\}_{z\in\Z}$  with $\mathsf x_z=  x_z$ for  $z=1, \ldots, L$.
Moreover, for all $\tau\in \R$ we denote by $X+\tau$  the element of $\X^ L$ whose $ L$-periodic extension is $\XX + \tau$. 

We denote by $\{e_{ l}\}_{ l=1,\ldots,L}$ the canonical basis of $\R^L$ and we set 
\begin{equation}\label{trasla}
e[L]:=\sum_{ l=1}^{L}e_ l.
\end{equation}

Let $X\in\X^L$.
We denote by  $\mu^{X}:=\sum_{z \in \Z} \delta_{\mathsf x_z}$ the empirical measure associated to the configuration $X$
and by $u^X$ the (unique) function in $BV_{\rm loc}(\R)$ satisfying  (in the distributional sense) 
\begin{equation}\label{misura}
\rm D u=\ud x- \mu^X
\end{equation}
and 
\begin{equation}\label{zeroaverage}
\int_{0}^{L}u^X\,\ud x=0.
\end{equation}
Note that, in view of \eqref{misura},
\begin{equation}\label{linftyest}
u^X\in L^\infty(\R).
\end{equation}
\subsection{The energy functionals}
Let $-\frac 1 2<\sigma<\frac 1 2$. 
For every $X\in\X^L$ and for every $x\in (0,L)\setminus\sss(X)$
we define
\begin{equation}\label{fraclap0}
(-\Delta)^{\sigma}u^X(x):=\left\{
\begin{array}{ll}
\displaystyle \int_{\R}\frac{u^X(y)}{|x-y|^{1+2\sigma}}\,\ud y&\textrm{if }-\frac{1}{2}<\sigma<0\\
\\
u^X(x)&\textrm{if }\sigma=0\\
\\
\displaystyle \int_{\R}\frac{u^X(x)-u^X(y)}{|x-y|^{1+2\sigma}}\,\ud y&\textrm{if }0<\sigma<\frac 1 2.
\end{array}
\right.
\end{equation}
In the formula \eqref{fraclap0} above the first integral is meant in the sense of principal value at infinity, i.e., as the limit, as $M\to +\infty$, of the integral on the intervals $(-M,M)$; it is easy to check that, as a consequence of periodicity and of the zero average condition \eqref{zeroaverage},  such a limit exists.  

Lemma \ref{welldef} below clarifies the summability of $(-\Delta)^\sigma u^X$.
\begin{lemma}\label{welldef}
Let $X\in\X^L$. Then,  $(-\Delta)^\sigma u^X\in L^1((0,L))$ for every $-\frac 1 2<\sigma<\frac 1 2$. Moreover, $(-\Delta)^\sigma u^X\in L^\infty(\R)$ whenever $-\frac 1 2<\sigma\le 0$.
\end{lemma}
\begin{proof}
By \eqref{linftyest}, we immediately have  that $(-\Delta)^0u^X\in L^\infty(\R)$.

Let now $-\frac 1 2<\sigma<0$. 
Then, for every $x,y\in (0,L)$ and for every $z\in\Z\setminus[-1,2]$, by  Lagrange Theorem  we have
\begin{equation*}
|x-y+zL|^{-2\sigma-1}-|z L|^{-2\sigma-1}=(-2\sigma-1)|\xi_{x,y} + zL|^{-2\sigma-2}\frac{\xi_{x,y}+ z L}{|\xi_{x,y}+ z L|}(x-y)
\end{equation*}
for some $\xi_{x,y}\in(-|x-y|,|x-y|)$.
Hence, for every $x\in (0,L)$, using \eqref{zeroaverage} and \eqref{linftyest}  
\begin{equation}\label{infifin}
\begin{aligned}
&\,\int_{2L}^{+\infty}\frac{u^X(y)}{|x-y|^{1 + 2\sigma}}\,\ud y=\sum_{n=2}^{+\infty} \int_{0}^{L}\frac{u^X(y)}{|x-y - nL|^{1 + 2 \sigma}}\,\ud y\\
=&\,\sum_{n=2}^{+\infty} \int_{0}^{L} u^X(y) (-2\sigma-1)|\xi_{x,y} - nL|^{-2\sigma-2}\frac{\xi_{x,y} - nL}{|\xi_{x,y} -nL|}(x-y)\,\ud y\le C_L,
\end{aligned}
\end{equation}
for some constant $C_L$ depending only on $L$. 
By the same reasoning,
\begin{equation}\label{infifin2}
\begin{aligned}
\int_{- \infty}^{- L}\frac{u^X(y)}{|x-y|^{1 + 2\sigma}}\,\ud y\le C_L.
\end{aligned}
\end{equation}
Therefore, since also
$$
\int_{-L}^{2L}\frac{u^X(y) }{|x-y|^{1 + 2 \sigma}} \ud y \le C_{L},
$$ 
by \eqref{infifin} and \eqref{infifin2}, we get that $(-\Delta)^\sigma u^X\in L^\infty((0,L))$ and hence, by periodicity,  $(-\Delta)^\sigma u^X\in L^\infty(\R)$.

Finally, we consider the case $0<\sigma<\frac 12$.
 Let $x\in (0,L)\setminus\sss(X)$ and let 
 $$\delta(x):=\min_{\xi\in\sss(X)}\min_{z\in\Z}|x-\xi+Lz|.
 $$
  Then
\begin{equation*}
\begin{aligned}
 |(-\Delta)^{\sigma}u^X(x)|\le&\, \int_{x-\delta(x)}^{x+\delta(x)}\frac{|u^X(x)-u^X(y)|}{|x-y|^{1+2\sigma}}\,\ud y+2\|u\|_{L^\infty} \int_{\R\setminus(x-\delta(x),x+\delta(x))}\frac{\ud y}{|x-y|^{1+2\sigma}} \\
\le&\, \frac{2\delta^{1-2\sigma}(x)}{1-2\sigma}+\frac{2}{\sigma}\|u\|_{L^\infty}\delta^{-2\sigma}(x),
\end{aligned}
\end{equation*}
so that, for every $p<\frac{1}{2\sigma}$,
\begin{equation}\label{contoprec}
\int_{0}^{L}|(-\Delta)^{\sigma}u^X(x)|^p\,\ud x\le C_{\sigma,L}+C_\sigma\|u\|_{L^\infty}\int_{0}^{L}\delta^{-2\sigma p}(x)\,\ud x<+\infty.
\end{equation}
Since $\sigma<\frac 1 2$, \eqref{contoprec} holds true with $p=1$.  
This concludes the proof.
\end{proof}
\begin{remark}
\rm{
In the proof of Lemma \ref{welldef} we never used the zero average assumption \eqref{zeroaverage} for the case $0\le \sigma<\frac 1 2$.
Nevertheless, the condition \eqref{zeroaverage} is crucial for proving the claim in the case
 $-\frac 1 2<\sigma<0$.
 }
 \end{remark}
We define the energy functional $\E^\sigma:\X^L\to\R$ as
\begin{equation}\label{ene}
 \E^{\sigma}(X)=\frac12  \langle u^X,(-\Delta)^\sigma u^X\rangle_{(0,L)},
\end{equation}
where the duality is meant in the sense $L^\infty((0,L))-L^1((0,L))$.
By construction, $\E^\sigma$ is invariant under to global translations, i.e., $\E^\sigma(X+\tau)=\E^{\sigma}(X)$ for every $X\in\X^L$ and for every $\tau\in\R$.
\begin{remark}\label{inverl2}
\rm{In view of Lemma \ref{welldef}, for $-\frac 1 2<\sigma\le 0$, the duality defining $\E^\sigma$ in \eqref{ene} can be understood in the $L^2$-sense. In view of \eqref{contoprec}, this holds true also for $0<\sigma<\frac 1 4$.
}
\end{remark}
The main result of this paper is the following theorem.
\begin{theorem}\label{mainthm}
For every $L\in\N$ and for every $-\frac 1 2<\sigma<\frac 1 2$ the unique, up to global translations and permutations of the indices, minimizer of $\E^\sigma$ in $\X^L$ is given by 
\begin{equation}\label{periocon}
\overline X^L=(0,1,\ldots,L-1).
\end{equation}
\end{theorem}
\begin{remark}
\rm{
It is easy to check that
\begin{equation}\label{fraclapi}
\E^\sigma(X)=\left\{
\begin{array}{ll}
\displaystyle  \frac 1 2\int_{0}^{L}\ud x\int_{\R}\frac{u^X(x)u^X(y)}{|x-y|^{1+2\sigma}}\,\ud y&\textrm{if }-\frac{1}{2}<\sigma<0\\
\\
\displaystyle \frac 1 2 \|u^X\|^2_{L^2((0,L))}&\textrm{if }\sigma=0\\
\\
\displaystyle \frac 1 4\int_{0}^{L}\,\ud x\int_{\R}\frac{|u^X(x)-u^X(y)|^2}{|x-y|^{1+2\sigma}}\,\ud y&\textrm{if }0<\sigma<\frac 1 2.
\end{array}
\right.
\end{equation}
We notice that first variations of the energy functionals $\E^\sigma$ with respect to $u$ coincide with the Laplacian operators defined in \eqref{fraclap0}. 
}
\end{remark}
%
\section{Fourier approach}
This section is devoted to the proof of Proposition \ref{persigma} below, which consists in rewriting the energy $\E^\sigma(X)$ as a function of the mutual (periodic) distances between the points in $\sss(X)$.
\begin{proposition}\label{persigma}
Let $-\frac 1 2 <\sigma<\frac 1 2$. Then, for every $X\in\X^L$,
\begin{equation}\label{risfou}
\E^\sigma(X)=C_{\sigma,L}\sum_{k\in\Z\setminus\{0\}}|k|^{2\sigma-2}\sum_{l,l'=1}^{L}\cos\Big(\frac{2\pi}{L}k\big(x_l-x_{l'}\big)\Big),
\end{equation}
where 
\begin{equation}\label{csenzcapp}
C_{\sigma,L}:=\frac{L}{8\pi^2}  \widehat C_{\sigma,L}
\end{equation}
and
\begin{equation}\label{ccapp}
\widehat C_{\sigma,L}:=\left\{
\begin{array}{ll}
 \displaystyle \frac{\pi^{2\sigma+\frac 1 2}}{L^{2\sigma}}\frac{\Gamma(-\sigma)}{\Gamma\big(\sigma+\frac 1 2\big)}   &\textrm{if }-\frac 1 2<\sigma<0,\\
 \displaystyle1&\textrm{if }\sigma=0,\\
 \displaystyle \frac{2^{2\sigma}\pi^{2\sigma}}{L^{2\sigma}}\frac{1}{\sigma}\Gamma(1-2\sigma)\cos(\pi\sigma)  &\textrm{if }0<\sigma<\frac 1 2.
\end{array}
\right.
\end{equation} 
\end{proposition}
Here and below the symbol $\Gamma$ denotes the
Euler $\Gamma$-function  defined for every complex number $\mathsf{z}\in\C$ with $\mathfrak{Re}[\mathsf{z}]>0$ as
\begin{equation}\label{gammafun}
\Gamma(\mathsf{z}):=\int_{0}^{+\infty}\tau^{\mathsf{z}-1}e^{-\tau}\,\ud \tau.
\end{equation}
In order to prove Proposition \ref{persigma}, we introduce some basic notions of Fourier analysis. Let $\mu$ be a scalar Radon measure on $\R$ that is $L$-periodic, i.e., such that $\mu(I+L)=\mu(I)$ for every interval $I$.
We can define the Fourier coefficients $\F[\mu](k)$ of $\mu$ as
\begin{equation}\label{fouru}
\F[\mu](k):=\langle \phi_k,\mu\res[0,L)\rangle=\int_{0}^L\phi_k \,\ud \mu,
\end{equation}
where $\phi_k(x):=L^{-\frac 1 2}e^{-\frac{2\pi}{L}ikx}$ for every $k\in\Z$.

In Lemma \ref{fourfraclap} below, we determine the Fourier coefficients of $(-\Delta)^\sigma u^X$, that, in view of its definition \eqref{fraclap0} and of Lemma \ref{welldef}, is an $L$-periodic function in $L^1_{\loc}(\R)$.
\begin{lemma}\label{fourfraclap}
Let $-\frac 1 2<\sigma<\frac 1 2$.
For every $X\in \X^L$ we have that
\begin{equation}\label{trasfdelta}
\F[(-\Delta)^\sigma u^X](k)=\widehat C_{\sigma,L} |k|^{2\sigma}\F[u^X](k)\qquad\textrm{for every }k\in\Z\setminus\{0\},
\end{equation}
where $\widehat C_{\sigma,L}$ is the constant defined in \eqref{ccapp}.

Moreover,
\begin{equation}\label{trasf0<}
\F[(-\Delta)^\sigma u^X](0)=0.
\end{equation}
\end{lemma}
\begin{proof}
For $\sigma=0$, the claims \eqref{trasfdelta} and \eqref{trasf0<} are trivial consequences of \eqref{fraclap0} and \eqref{zeroaverage}.
\vskip5pt
Now we discuss the case $-\frac 1 2<\sigma<0$. For $k=0$, we have
 \begin{equation}\label{perk0}
\F[(-\Delta)^\sigma u^X](0)=L^{-\frac 1 2}\int_{0}^{L}\ud x\int_{\R}\frac{u^X(y)}{|x-y|^{1+2\sigma}}\,\ud y=L^{-\frac 1 2}\int_{\R}\frac{\ud z}{|z|^{1+2\sigma}}\int_{0}^{L}u^{X}(x+z)\,\ud x=0,
\end{equation}
where in the second equality we used the change of variable $z=y-x$ and exchanged the integrals, and the last equality is a consequence of \eqref{zeroaverage}.
This proves \eqref{trasf0<} for $-\frac 1 2<\sigma<0$.

Let now $k\in\Z\setminus\{0\}$.  In analogy with \eqref{perk0}, using (twice) the change of variable $z=y-x$, we have
\begin{equation}\label{perknon0}
\begin{aligned}
\F[(-\Delta)^\sigma u^X](k)=&\,L^{-\frac 1 2}\int_{0}^{L}\ud x\, e^{-\frac{2\pi}{L}ikx}\int_{\R}\frac{u^X(y)}{|x-y|^{1+2\sigma}}\,\ud y\\
=&\,L^{-\frac 1 2}\int_{\R}\ud z\, \frac{e^{\frac{2\pi}{L}ikz}}{|z|^{1+2\sigma}}\int_{0}^{L}u^X(x+z)e^{-\frac{2\pi}{L}i k(x+z)}\,\ud x\\
=&\,L^{-\frac 1 2}\int_{\R}\ud z\,\frac{e^{\frac{2\pi}{L}ikz}}{|z|^{1+2\sigma}}\int_{z}^{L+z}u^X(y)e^{-\frac{2\pi}{L}i ky}\,\ud y\\
=&\,\F[u^X](k)\int_{\R} \frac{e^{\frac{2\pi}{L}ikz}}{|z|^{1+2\sigma}}\,\ud z.
\end{aligned}
\end{equation}
Therefore, \eqref{trasfdelta}, for $-\frac 1 2<\sigma<0$, is a consequence of the following identity 
\begin{equation}\label{fourK}
\int_{\R} \frac{e^{\frac{2\pi}{L}ikz}}{|z|^{1+2\sigma}}\,\ud z=
\frac{\pi^{2\sigma+\frac 1 2}}{L^{2\sigma}}\frac{\Gamma(-\sigma)}{\Gamma\big(\sigma+\frac 1 2\big)}
|k|^{2\sigma},
\end{equation}
which is a well-established fact in Fourier analysis.
For the reader's convenience, we provide here its short proof.
To prove \eqref{fourK}, we preliminarily notice that, for every $u,\beta>0$, using the change of variable $\tau=ut$ in \eqref{gammafun}, we obtain
\begin{equation}\label{gammafun2}
u^{-\beta}=\frac{1}{\Gamma(\beta)}\int_{0}^{+\infty}t^{\beta-1}e^{-ut}\,\ud t. 
\end{equation}
Applying \eqref{gammafun2} with $\beta=\sigma+\frac{1}{2}$ and $u=z^2$ and exchanging the order of integrals, we have
\begin{equation}\label{pri}
\begin{aligned}
\int_{\R} \frac{e^{\frac{2\pi}{L}ikz}}{|z|^{1+2\sigma}}\,\ud z
=&\,\frac{1}{\Gamma\big(\sigma+\frac 1 2\big)}\int_{0}^{+\infty}\ud t \,t^{\sigma-\frac 1 2}\int_{\R}e^{-tz^2+\frac{2\pi}{L}ikz}\,\ud z\\
=&\,\frac{1}{\Gamma\big(\sigma+\frac 1 2\big)}\int_{0}^{+\infty}\ud t \,t^{\sigma-\frac 1 2} e^{-\frac{\pi^2}{L^2}\frac{k^2}{t}}\int_{\R}e^{-\big(z\sqrt t-i\frac{\pi}{L}\frac{k}{\sqrt t}\big)^2}\,\ud z.
\end{aligned}
\end{equation} 
Now, by the change of variable $w=z\sqrt t-i\frac{\pi}{L}\frac{k}{\sqrt t}$, we have that
\begin{equation*}
\int_{\R}e^{-\big(z\sqrt t-i\frac{\pi}{L}\frac{k}{\sqrt t}\big)^2}\,\ud z=t^{-\frac 1 2}\int_{\R}e^{-w^2}\,\ud w=\sqrt{\pi} t^{-\frac 12}, 
\end{equation*}
which, replaced in the righthand side of \eqref{pri}, yields
\begin{equation*}
\begin{aligned}
\int_{\R} \frac{e^{\frac{2\pi}{L}ikz}}{|z|^{1+2\sigma}}\,\ud z=&\,\frac{\sqrt{\pi}}{\Gamma\big(\sigma+\frac 1 2\big)}\int_{0}^{+\infty}t^{\sigma-1} e^{-\frac{\pi^2}{L^2}\frac{k^2}{t}}\,\ud t\\
=&\,\frac{\sqrt{\pi}}{\Gamma\big(\sigma+\frac 1 2\big)}\frac{\pi^{2\sigma}}{L^{2\sigma}}|k|^{2\sigma}\int_{0}^{+\infty}\tau^{-\sigma-1} e^{-\tau}\,\ud \tau\\
=&\,\frac{\pi^{2\sigma+\frac 1 2}}{L^{2\sigma}}\frac{\Gamma(-\sigma)}{\Gamma\big(\sigma+\frac 1 2\big)}
|k|^{2\sigma},
\end{aligned}
\end{equation*}
where last but one equality follows from the change of variable $\tau=\frac{\pi^2}{L^2}\frac{k^2}{t}$ and the last one is an immediate consequence of the definition in \eqref{gammafun} (applied with $\mathsf z=-\sigma$).
This concludes the proof of \eqref{fourK} and, in turns, that of \eqref{trasfdelta} in the case $-\frac 12 <\sigma<0$.
\vskip5pt
We conclude by discussing the case $0<\sigma<\frac 12$.
To this end, we notice that, for every $x\in(0,L)\setminus\sss(X)$
\begin{equation}\label{nopv}
(-\Delta)^\sigma u^X(x)= \frac12 \int_{\R}\frac{2u^X(x)-u^X(x+z)-u^X(x-z)}{|z|^{1+2\sigma}}\,\ud z.
\end{equation}
Hence, in particular, by exchanging the order of integrals and using the periodicity of $u^X$, 
we have
\begin{equation*}
\F[(-\Delta)^\sigma u^X](0)= \frac{L^{-\frac 1 2}}{2}  \int_{\R}\frac{\ud z}{|z|^{1+2\sigma}}\int_{0}^{L}\big(2u^X(x)-u^{X}(x+z)-u^{X}(x-z)\big)\ud x=0,
\end{equation*}
which proves \eqref{trasf0<} also in this case.
Moreover, by \eqref{nopv}, for every $k\in\Z\setminus\{0\}$, we have
\begin{equation}\label{mancagamma}
\begin{aligned}
\F[(-\Delta)^\sigma u^X](k)=&\, \frac{L^{-\frac 1 2}}{2}
\int_{0}^{L}\ud x \,e^{-\frac{2\pi}{L}ikx}
\int_{\R}\frac{2u^X(x)-u^X(x+z)-u^X(x-z)}{|z|^{1+2\sigma}}\,\ud z
\\
=&\, \frac{L^{-\frac 1 2}}{2}\int_{\R}\frac{\ud z}{|z|^{1+2\sigma}}\int_{0}^{L}(2u^X(x)-u^X(x+z)-u^X(x-z)) e^{-\frac{2\pi}{L}ikx}\,\ud x\\
=&\,\F[u^X](k) \frac12 \int_{\R}\frac{2-e^{\frac{2\pi}{L}ikz}-e^{-\frac{2\pi}{L}ikz}}{|z|^{1+2\sigma}} \,\ud z\\
=&\,\F[u^X](k) \int_{\R}\frac{1-\cos\big(\frac{2\pi}{L}kz)}{|z|^{1+2\sigma}} \,\ud z,
\end{aligned}
\end{equation}
where the second equality follows by exchanging the order of integrals and the third one is a consequence of the change of variable $y=x\pm z$ that allows to deduce 
\begin{equation*}
L^{-\frac 1 2}\int_{0}^{L}u(x\pm z)e^{-\frac{2\pi}{L}ikx}\,\ud x=e^{\pm\frac{2\pi}{L}ikz}\F[u^X](k).
\end{equation*}
Now, by the change of variable $t=\frac{2\pi}{L}kz$, integrating by parts, we have
\begin{equation}\label{pergamma}
\begin{aligned}
\int_{\R}\frac{1-\cos\big(\frac{2\pi}{L}kz\big)}{|z|^{1+2\sigma}} \,\ud z=&\,\Big(\frac{2\pi}{L}\Big)^{2\sigma}|k|^{2\sigma}\int_{\R}\frac{1-\cos t}{|t|^{1+2\sigma}}\,\ud t\\
=&\,\frac{2^{2\sigma}\pi^{2\sigma}}{L^{2\sigma}}|k|^{2\sigma}\frac{1}{\sigma}\int_{0}^{+\infty}t^{-2\sigma}\sin t\,\ud t\\
=&\,\frac{2^{2\sigma}\pi^{2\sigma}}{L^{2\sigma}}\frac{1}{\sigma}\Gamma(1-2\sigma)\cos(\pi\sigma)|k|^{2\sigma},
\end{aligned}
\end{equation}
where in the last equality we applied the classical formula (see \cite[formula 3.761.4]{GR})
\begin{equation*}
\int_{0}^{+\infty}x^{\alpha-1}\sin x\,\ud x=\Gamma(\alpha)\sin\big(\frac{\pi\alpha}{2}\big)\qquad\textrm{for }0<\mathfrak{Re}[\alpha]<1,
\end{equation*}
with $\alpha=1-2\sigma\in(0,1)$.
By inserting \eqref{pergamma} in \eqref{mancagamma}
we obtain \eqref{trasfdelta} also for $0<\sigma<\frac 1 2$. The proof is concluded.
\end{proof}
\begin{remark}
\rm{
We notice that, for $0<\sigma < \frac12$,  the proof of formula \eqref{trasfdelta} works actually for every locally Lipschitz $L$-periodic function $u\in L^\infty(\R)$ whose fractional Laplacian (defined as in \eqref{fraclap0}) is in $L^1_{\loc}(\R)$.
Moreover, since $\lim_{\mathsf z\to 0^+}\mathsf{z}\Gamma(\mathsf z)=1$, $\Gamma(\frac 1 2)=\sqrt{\pi}$ and $\Gamma(1)=1$, we have that
\begin{equation*}
\lim_{\sigma\to 0}   |\sigma|\widehat C_{\sigma,L}= 1 ,
\end{equation*}
which, in view of \eqref{trasfdelta}, implies 
\begin{equation*}
\lim_{\sigma\to 0}   |\sigma|\F[(-\Delta)^\sigma u^X](k)=  \F[(-\Delta)^0 u^X](k)\qquad\textrm{for every }k\in\Z\setminus\{0\}.
\end{equation*}
}
\end{remark}
Now we compute the Fourier coefficients of $u^X$.
\begin{lemma}\label{fourieru}
Let $L\in\N$. For every $X\in\X^L$ we have that $\F[u^X](0)\equiv 0$ and
\begin{equation*}
\F[u^X](k)=i\frac{L^{\frac 1 2}}{2\pi k}\sum_{l=1}^{L}e^{-\frac{2\pi}{L}ikx_l} \qquad\textrm{for every }k\in\Z\setminus\{0\} .
\end{equation*}
\end{lemma}
\begin{proof}
By  \eqref{zeroaverage}, $\F[u^X](0)=L^{-\frac 1 2}\int_{0}^{L}u^X\ud x=0$. Let $k\in\Z\setminus\{0\}$. Then, 
\begin{equation*}
i\frac{2\pi}{L}k\F[u^{X}](k)=\F[\ud x-\mu^X](k)=\int_{0}^{L}\phi_k(x)\,\ud x-\int_{0}^{L}\phi_k(x)\,\ud\mu^X=-L^{-\frac 1 2}\sum_{l=1}^{L}e^{-\frac{2\pi}{L}ikx_l}.
\end{equation*}
\end{proof}
With Lemmas \ref{fourfraclap} and \ref{fourieru} in hand, we can finally prove Proposition \ref{persigma}.
\begin{proof}[Proof of Proposition \ref{persigma}]
We first focus on the case $-\frac 1 2<\sigma<\frac 1 4$.
In view of Remark \ref{inverl2}, in this case the duality defining $\E^\sigma$ in \eqref{ene} can be understood in the $L^2$-sense.
Therefore, by Parseval identity, Lemma \ref{fourfraclap}, Lemma \ref{fourieru} and \eqref{csenzcapp}, we have
\begin{equation}\label{conparse}
\begin{aligned}
\E^\sigma(X)=&\,
\frac12 \sum_{k\in\Z}\overline{\F[u^X](k)}\F[(-\Delta)^\sigma u^X](k)=\frac12 \widehat C_{\sigma,L}\sum_{k\in\Z\setminus\{0\}}|k|^{2\sigma}\big|\F[u^X](k)\big|^2\\
=&\,
\frac{L}{8\pi^2} \widehat C_{\sigma,L}\sum_{k\in\Z\setminus\{0\}}|k|^{2\sigma-2}\sum_{l,l'=1}^{L}\cos\Big(\frac{2\pi}{L}k\big(x_l-x_{l'}\big)\Big)\\
=&\,C_{\sigma,L}\sum_{k\in\Z\setminus\{0\}}|k|^{2\sigma-2}\sum_{l,l'=1}^{L}\cos\Big(\frac{2\pi}{L}k\big(x_l-x_{l'}\big)\Big),
\end{aligned}
\end{equation}
which proves \eqref{risfou} in this case.

Finally, we discuss the case $\frac 1 4\le \sigma<\frac 1 2$.
Let $\rho$ be a standard mollifier and set $\rho_\ep(\cdot):=\frac 1 \ep\rho\big(\frac \cdot \ep\big)$ for every $0<\ep<1$. 
Let $v\in L^\infty(\R)$ be an $L$-periodic function.
Then, $v\ast\rho_\ep$ is a $C^\infty$ and $L$-periodic function on $\R$.
By arguing as in the proof of \eqref{perknon0}, one can show that for every $k\in\Z$
\begin{equation}\label{convoF}
\F[v\ast\rho_\ep](k)=\F[v](k)\int_{\R}\rho_\ep(z)e^{\frac{2\pi}{L}ikz}\,\ud z\qquad\textrm{for every }k\in\Z.
\end{equation}
Let now $X\in \X^L$.
By applying \eqref{convoF} with $v=u^X$
we have that
\begin{equation}\label{convo0}
\F[u^X\ast\rho_\ep](k)=\F[u^X](k)\int_{\R}\rho_\ep(z)e^{\frac{2\pi}{L}ikz}\,\ud z.
\end{equation}
Moreover, since $(-\Delta)^\sigma (u^X\ast\rho_\ep)=((-\Delta)^\sigma u^X)\ast\rho_\ep$, by applying \eqref{convoF} with $v=(-\Delta)^\sigma u^X$, for every $k\in\Z$ we have
\begin{equation*}
\F[(-\Delta)^\sigma(u^X\ast\rho_\ep)](k)=\F[(-\Delta)^\sigma u^X](k)\int_{\R}\rho_\ep(z)e^{\frac{2\pi}{L}ikz}\,\ud z,
\end{equation*}
which, in view of Lemma \ref{fourfraclap} (in particular, by \eqref{trasfdelta}), implies 
\begin{equation}\label{convodelta}
\F[(-\Delta)^\sigma(u^X\ast\rho_\ep)](k)=
\widehat C_{\sigma,L}|k|^{2\sigma}\,\F[u^X](k)\int_{\R}\rho_\ep(z)e^{\frac{2\pi}{L}ikz}\,\ud z\quad\textrm{for }k\in\Z\setminus\{0\}.
\end{equation}
Therefore, in view of \eqref{convo0} and \eqref{convodelta}, by arguing as in \eqref{conparse}, for every $\ep>0$, we have
\begin{equation}\label{conappro}
\begin{aligned}
&\,\frac12 \langle u^X\ast\rho_\ep, (-\Delta)^\sigma (u^X\ast\rho_\ep)\rangle_{(0,L)}\\
=&\,\frac12 \sum_{k\in\Z}\overline{\F[u^X](k)}\F[(-\Delta)^\sigma u^X](k)\Big|\int_{\R}\rho_\ep(z)e^{\frac{2\pi}{L}ikz}\,\ud z\Big|^2\\
=&\,C_{\sigma,L}\sum_{k\in\Z\setminus\{0\}}|k|^{2\sigma-2}\sum_{l,l'=1}^{L}\cos\Big(\frac{2\pi}{L}k\big(x_l-x_{l'}\big)\Big)\Big|\int_{\R}\rho_\ep(z)e^{\frac{2\pi}{L}ikz}\,\ud z\Big|^2.
\end{aligned}
\end{equation}
Since $2\sigma-2<-1$, we have that the series on the righthand side of \eqref{conappro} converges absolutely, so that, sending $\ep\to 0$ in \eqref{conappro}, we get \eqref{risfou} also in this case.

The proof is concluded. 
\end{proof}
\section{Proof of Theorem \ref{mainthm}}
The key ingredient in order to prove Theorem \ref{mainthm} is the celebrated result  by Cohn and Kumar in \cite[Theorem 1.2]{CK}.
We recall it here only in dimension $2$, that is the case involved in our problem. We recall that a $C^\infty$ function $V:I\to \R$ is {\it strictly  completely monotone} on the interval $I$ if 
\begin{equation}\label{complemon}
(-1)^j\frac{\ud^j}{\ud \zeta^j}V(\zeta)\ge 0,\qquad\textrm{for every }j\in\N, \quad \zeta\in I.
\end{equation}
Moreover, $V$ is {\it strictly completely monotone} if the inequality in \eqref{complemon} is strict in the interior of $I$.
For every $L\in\N$ we set
\begin{equation*}
\mathcal{C}^L:=\big\{C=(z_1,\ldots,z_L)\,:\,z_l\in \mathbb{S}^1\textrm{ for every }l=1,\ldots,L,\quad z_l\neq z_{l'}\textrm{ for }l\neq l'\big\}.
\end{equation*}
\begin{theorem}[Cohn-Kumar]\label{cohnkumar}
Let $V:(0,4]\to (0,+\infty)$ be a strictly completely monotone function. Then, for every $L\in\N$ the unique, up to rotations and permutations of the indices, minimizer of the energy
\begin{equation*}
\mathscr{H}_V(C):=\sum_{\newatop{l,l'=1}{l\neq l'}}^{L}V(|z_l-z_{l'}|^2)
\end{equation*}
in $\mathcal{C}^L$ is the regular $L$-gon $\overline{C}_L:=(1, e^{\frac{2\pi}{L}i}, \ldots e^{\frac{2\pi}{L}i(L-1)})$.
\end{theorem}
Now we provide the proof of Theorem \ref{mainthm}.
Let $-\frac{1}{2}<\sigma<\frac 12$ and let $X\in \X^L$. 
By Proposition \ref{persigma}, applying \eqref{gammafun2} with $\beta=2-2\sigma$ and $u=|k|$, and exchanging the order of summation, we have
\begin{equation}\label{intermstep}
\E^\sigma(X)=\frac{C_{\sigma,L}}{\Gamma(2-2\sigma)}\int_{0}^{+\infty} t^{1-2\sigma}\sum_{l,l'=1}^{L}\sum_{k\in\Z\setminus\{0\}}e^{-|k|t}\cos\Big(\frac{2\pi}{L}k(x_l-x_{l'})\Big)\,\ud t.
\end{equation}
Using the identity
\begin{equation*}
\sum_{k\in\Z} r^{|k|}\cos(k\theta)=\frac{1-r^2}{1-2r\cos\theta+r^2}=\frac{1-r^2}{(1-r)^2+2r(1-\cos\theta)},
\end{equation*}
with $r=e^{-t}$ and $\theta=\frac{2\pi}{L}(x_l-x_{l'})$, for every $l,l'=1,\ldots,L$ we obtain
\begin{equation}\label{conpoi}
\sum_{k\in\Z\setminus\{0\}}e^{-|k|t}\cos\Big(\frac{2\pi}{L}k(x_l-x_{l'})\Big)=\frac{1-e^{-2t}}{(1-e^{-t})^2+2e^{-t}\Big(1-\cos\Big(\frac{2\pi}{L} (x_l-x_{l'})\Big)\Big)}-1.
\end{equation}
Since
\begin{equation}\label{ins1}
\big|e^{\frac{2\pi}{L}ix_l}-e^{\frac{2\pi}{L}ix_{l'}}\big|^2=2\Big(1-\cos\Big(\frac{2\pi}{L}(x_l-x_{l'})\Big)\Big)\qquad\textrm{for every }l,l'=1,\ldots,L,
\end{equation}
by \eqref{conpoi}, we get
\begin{equation}\label{intermstep1}
\sum_{k\in\Z\setminus\{0\}}e^{-|k|t}\cos\Big(\frac{2\pi}{L}k(x_l-x_{l'})\Big)=\frac{1-e^{-2t}}{(1-e^{-t})^2+e^{-t}|e^{\frac{2\pi}{L}ix_l}-e^{\frac{2\pi}{L}ix_{l'}}|^2}-1.
\end{equation}
Hence, setting, for every $t>0$, 
\begin{equation}\label{potential}
V_t(\zeta):=\frac{1-e^{-2t}}{(1-e^{-t})^2+e^{-t}\zeta},\qquad (\zeta\ge 0)
\end{equation}
and
\begin{equation}\label{applcohnkumar}
\G_t(X):=\sum_{\newatop{l,l'=1}{l\neq l'}}^{L}V_t\big(\big|e^{\frac{2\pi}{L}ix_l}-e^{\frac{2\pi}{L}ix_{l'}}\big|^2\big),
\end{equation}
by \eqref{intermstep} and \eqref{intermstep1}, we obtain
\begin{equation}\label{energiaintegrata}
\E^\sigma(X)=\frac{C_{\sigma,L}}{\Gamma(2-2\sigma)}\int_{0}^{+\infty} t^{1-2\sigma}\Big(\G_t(X)+L\frac{1+e^{-t}}{1-e^{-t}}-L^2\Big)\,\ud t.
\end{equation}
Let $\mathfrak{b}:\rs^L\to \mathcal{C}^L$ be the bijection defined by $\mathfrak{b}(X):=(e^{\frac{2\pi}{L}ix_1},\ldots,e^{\frac{2\pi}{L}ix_L})$.
Trivially, $\mathfrak{b}(\overline X^L)=\overline{C}^L$.
Moreover, for every $X\in\rs^L$ we have that $\G_t(X)=\mathscr{H}_{V_t}(\mathfrak{b}(X))$.
For every $t>0$, the potential $V_t$ defined in \eqref{potential} is strictly completely monotone in $[0,+\infty)$; hence, by Theorem \ref{cohnkumar}, $\overline X^L$ is the unique, up to global translations and permutations of the indices, minimizer of $\G_t$ in $\rs^L$, and, in view of \eqref{energiaintegrata}, the unique, up to global translations and permutations of the indices, minimizer of $\E^\sigma$ in $\rs^L$.
Moreover, by continuity, for every $t>0$, $\overline X^L$ is a minimizer of $\G_t$ in $\X^L$ and, in view of
\eqref{energiaintegrata}, 
 a minimizer of $\E^\sigma$ in $\X^L$. 

Now we show that there are no further  minimizers of  $\E^\sigma$ 
 in $\X^L\setminus\rs^L$.
Assume by contradiction that there exists $\widehat X^L\in \X^L\setminus\rs^L$ such that $\E^\sigma(\widehat X^L)=\E^\sigma(\overline X^L)$. Then, by \eqref{energiaintegrata} and by the mean value theorem, for every $t>0$, $\widehat X^L$ minimizes $\G_t$ in $\X^L$ and hence $\G_t(\widehat X^L)=\G_t(\overline X^L)$. Moreover, since for every $\zeta>0$
$$
\lim_{t\to 0^+}\frac{1}{1-e^{-2t}}V_t(\zeta)=\frac{1}{\zeta}, 
$$
 we get that 
 $$
+\infty =\lim_{t\to 0^+}\frac{1}{1-e^{-2t}}\G_t(\widehat X^L)= \lim_{t\to 0^+}\frac{1}{1-e^{-2t}}\G_t(\overline X^L)<+\infty,
 $$
 thus providing a contradiction. This concludes the proof of Theorem \ref{mainthm}.

\end{document}